\documentclass[12pt,leqno]{article}   
\usepackage{amsmath}
\usepackage{amsfonts}
\usepackage{amsthm}
\usepackage{latexsym}
\usepackage{amssymb}
\usepackage{verbatim}
\usepackage{enumerate}
\usepackage{array}
\font\logic=msam10 at 10pt
\newcommand{\forces}{\mbox{\logic\char'015}}
\newcommand{\restrict}{\mbox{\logic\char'026}}
\def\underTilde#1{{\baselineskip=0pt\vtop{\hbox{$#1$}\hbox{$\sim$}}}{}}
\def\undertilde#1{{\baselineskip=0pt\vtop
  {\hbox{$#1$}\hbox{$\scriptscriptstyle\sim$}}}{}}

\newcommand{\lbrak}{[\mspace{-2.2mu}[}
\newcommand{\rbrak}{]\mspace{-2.2mu}]}

\newcommand{\truth}[1]{\lbrak #1 \rbrak}

\newcommand{\creals}{{^{\omega}}2}

\newtheorem{thrm}{Theorem}[section]
\newtheorem{lem}[thrm]{Lemma}

\newtheoremstyle{hdefinition}%
  {\topsep}%
  {\topsep}%
  {\upshape}
  {}%
  {\bfseries}%
  {.}
  { }%
  {\thmnumber{#2 }\thmname{#1}\thmnote{ \rm(#3)}}%

\newtheoremstyle{hclaim}%
  {\topsep}%
  {\topsep}%
  {\itshape}%
  {}%
  {\bfseries}%
  {.}
  { }%
  {\thmname{#1}\thmnote{ \rm#3}}%

\newtheoremstyle{hnotation}%
  {\topsep}%
  {\topsep}%
  {\upshape}%
  {}%
  {\bfseries}%
  {.}
  { }%
  {\thmname{#1}\thmnote{ \rm#3}}%

\theoremstyle{hclaim}
\newtheorem*{claim*}{Claim}

\theoremstyle{hdefinition}
\newtheorem{df}[thrm]{Definition}

\newtheorem{remark}[thrm]{Remark}
\newtheorem{ques}[thrm]{Question}

\theoremstyle{hclaim}
\newtheorem{claim}{Claim}

\theoremstyle{hnotation}
\newtheorem{notation}{Notation}

\begin{document}

\title{Universally measurable sets in generic extensions
\thanks{
The first author is supported in part by NSF grant DMS-0801009.
The research of the second author is supported by NSF grant DMS-0556223.
The research of the third author is supported by the United
States-Israel Binational Science Foundation.
The research in this paper began during a visit by the
first author to Rutgers University in October 2008, supported by NSF
grant DMS-0600940. Publication no. 947 of second author.}}

\author{Paul Larson \and Itay Neeman \and Saharon Shelah}


\pagenumbering{arabic}
\maketitle

\begin{abstract} A subset of a topological space is said to be \emph{universally measurable} if it is measured by the completion of each countably additive $\sigma$-finite Borel measure on the space, and \emph{universally null} if it has measure zero for each such atomless measure. In 1908, Hausdorff proved that there exist universally null sets of real numbers of cardinality $\aleph_{1}$, and thus that there exist at least $2^{\aleph_{1}}$ such sets. Laver showed in the 1970's that consistently there are just continuum many universally null sets of reals. The question of whether there exist more than continuum many universally measurable sets of reals was asked by Mauldin in 1978. We show that consistently there exist only continuum many universally measurable sets. This result also follows from work of Ciesielski and Pawlikowski on the iterated Sacks model. In the models we consider (forcing extensions by suitably-sized random algebras) every set of reals is universally measurable
if and only if it and its complement are unions of ground model continuum many Borel sets.
\end{abstract}


\vspace{\baselineskip}

A subset of a topological space is said to be \emph{universally measurable} if it is measurable with respect to every complete $\sigma$-finite Borel measure on the space, and \emph{universally null} if it has measure zero for each such atomless measure (see \cite{Ke, Ni}, and 434D of \cite{F4}, for instance).  Hausdorff \cite{H} proved that there exists a universally null set of reals of cardinality $\aleph_{1}$, which implies that there exist at least $2^{\aleph_{1}}$ such sets (more recently, Rec\l{}aw \cite{Re} has shown that every set of reals which is wellordered by a universally measurable relation is universally null). Laver (unpublished, see \cite{L}, pages 576--578 of \cite{M83} and Section 1.1 of \cite{CiPa}) showed that consistently there are just continuum many universally null sets of reals. The question of whether there exist more than continuum many universally measurable sets of reals was asked by Mauldin in 1978 (personal communication), though the question may not have appeared in print until 1984 (see \cite{M84, Ni}). We show that in a forcing extension by a suitable random algebra ($\mathbb{B}(\kappa)$ for $\kappa = \kappa^{\mathfrak{c}}$, where $\mathfrak{c}$ denotes the cardinal $2^{\aleph_{0}}$) there exist only continuum many universally measurable sets of reals, and moreover that
every set of reals in such an extension is universally measurable if and only if it and its complement are unions of ground model continuum many Borel sets.
We present two proofs of the consistent negative answer to Mauldin's question in the random algebra extension.

The negative answer to Mauldin's question also follows from work of Ciesielski and Pawlikowski on the axiom CPA in the iterated Sacks model. In Section \ref{CPA} we give a proof that their axiom CPA$_{cube}$ implies that every universally measurable set is the union of at most $\aleph_{1}$-many perfect sets and singletons, which is a very slight modification of their proof that CPA$_{cube}$ implies that every universally null set has cardinality less than or equal to $\aleph_{1}$. Moreover, their axiom CPA$^{game}_{cube}$ implies that the perfect sets can be taken to be disjoint (see Section 2.1 of \cite{CiPa}).


In the final section we discuss several open questions and other issues regarding universally measurable sets.

\begin{notation} We use the symbol $\subseteq$ to mean ``subset", and the symbol $\subset$ to mean ``proper subset."
\end{notation}

\section{Basic definitions and standard facts}

A \emph{measure} on a set $X$ is a function $\mu$ whose domain is some $\sigma$-algebra of subsets of $X$, with codomain $[0, \infty]$, such that $\mu$ is countably additive for disjoint families. A set is said to be \emph{measurable} with respect to $\mu$ if it is in the domain of $\mu$. A \emph{Borel measure} is a measure on a topological space whose domain contains the Borel sets. A measure is \emph{complete} if subsets of sets of measure $0$ are in the domain of the measure (and thus have measure 0). The \emph{completion} of a measure is the smallest complete measure extending it.  If $\mu$ is a Borel measure on a topological space $X$, and $\mu^{*}$ is the completion of $\mu$, then a set $A \subseteq X$ is in the domain of $\mu^{*}$ if and only if there is a set $B$ in the domain of $\mu$ such that the symmetric difference $A \bigtriangleup B$ is contained in a set of $\mu$-measure 0 (see 212C of \cite{F2}).
A measure $\mu$ on a set $X$ is a \emph{probability} measure if $\mu(X) = 1$, \emph{finite} if $\mu(X)$ is finite, $\sigma$-\emph{finite} if $X$ is a countable union of sets of finite measure, and \emph{atomless} if singletons have measure 0.

The set of universally measurable sets does not change if one replaces ``$\sigma$-finite" with ``finite" or ``probability" (see 211X(e) of \cite{F2}) or requires the measures to be atomless. It follows that universally null sets are universally measurable. Note that no perfect set is universally null, since
any such set has measure 1 for some Borel probability measure.

The following theorem (Theorem 15.6 of \cite{Ke}) shows among other things that the cardinality of the set of universally measurable subsets of any complete separable metric space (i.e., \emph{Polish} space) is the same, since any such space can be continuously injected into any other (see also Remark 2.9 of \cite{Ni}).

\begin{thrm}\label{isom} If $X$ and $Y$ are Polish spaces and $\mu$ and $\nu$ are atomless Borel probability measures on $X$ and $Y$
respectively, then there is a Borel bijection $f \colon X \to Y$  such that
$\mu(I) = \nu(f[I])$ for all Borel $I \subset X$.
\end{thrm}

Combining this fact with Lusin's theorem \cite{Lu} that analytic sets are Lebesgue measurable, we have that analytic sets are universally measurable.

\section{The random algebra and measures}\label{randmeas}

We review in this section some standard facts about the random algebra, and the way in which names in the random algebra give rise to measures.
Our primary references are \cite{Ku, BaJu, F5}.

\begin{notation} Fixing a nonempty set $X$, for each $x \in X$ and each $i \in 2$, we let $C^{i}_{x}$ denote the set $\{ f \in {^{X}}2 \mid f(x) = i\}$.
\end{notation}

\begin{df} The \emph{Baire} subsets of ${^{X}}2$ are the members of the smallest $\sigma$-algebra containing $\{ C^{i}_{x} \mid x \in X, \, i \in 2\}$.
\end{df}

As pointed out in \cite{Ku}, the classes of Baire sets and Borel sets coincide when $X$ is countable, but otherwise singletons are Borel but not Baire.


\begin{notation} We let $\mu_{X}$ denote the completion of the standard product measure on the Baire subsets of ${^{X}}2$, where $\mu_{X}(C^{i}_{x}) = 1/2$ for each $x \in X$ and $i \in 2$ (see 254J of \cite{F2}, for instance).
\end{notation}

\begin{notation} We let $\mathcal{N}_{X}$ denote the set $\{ B \subset {^{X}2 \mid \mu_X}(B) = 0\}$.
\end{notation}


\begin{df} For each Baire $B \subseteq {^{X}}2$, $[B]_{\mu_{X}}$ is the set of Baire $D \subseteq {^{X}}2$ such that $B \bigtriangleup D \in \mathcal{N}_{X}$. The \emph{random algebra} $\mathbb{B}(X)$ is the partial order whose conditions are sets of the form $[B]_{\mu_{X}}$, where $B$ is a
non-$\mu_{X}$-null Baire subset of ${^{X}}2$, with the order $$[B]_{\mu_{X}} \leq [D]_{\mu_{X}}$$ (i.e.,  $[B]_{\mu_{X}}$ is stronger than $[D]_{\mu_{X}}$) if and only if $B \setminus D \in \mathcal{N}_{X}$.
Noting that $\mu_{X}(C) = \mu_{X}(B)$ for all $C \in [B]_{\mu_{X}}$, we let $\mu_{X}([B]_{\mu_{X}})$ denote this common value.
\end{df}

Forcing with $\mathbb{B}(X)$ adds a \emph{generic function} $F \colon X \to 2$ defined by letting $F(x) = i$ if and only if $C^{i}_{x}$ is in the generic filter. We sometimes refer to the generic function $F$ without mentioning the corresponding filter $G$, and talk of the model $V[F]$.

\begin{notation}
Given $B \subseteq {^{X}}2$, $Y \subseteq X$ and $y \in {^{Y}}2$, we let $$B_{y} = \{ z \in {^{X \setminus Y}}2 \mid y \cup z \in B\}.$$
\end{notation}

Fubini's Theorem in this context says that for any nonempty $Y \subseteq X$ and any
Baire $B \subseteq {^{X}}2$, $B \in \mathcal{N}_{X}$ if and only if $\{ y \in {^{Y}}2 \mid B_{y} \not\in \mathcal{N}_{X\setminus Y} \} \in \mathcal{N}_{Y}$ (see \cite{Ku} or 252B and 254N of \cite{F2}).

\begin{df} Given $B \subseteq {^{X}}2$ and $Y \subseteq X$, we say that $B$ has \emph{support} $Y$ if for all $y \in {^{Y}}2$,
$B_{y} \in \{ {^{X\setminus Y}}2, \emptyset\}$ (so the support of $B$ is not unique).
A condition $[B]_{\mu_{X}}$ has \emph{support} $Y$ if some $B \in [B]_{\mu_{X}}$ has support $Y$. We say that $B$ has \emph{finite support} if $B$ has support $Y$ for some finite $Y$, and \emph{countable support} if it has support $Y$ for some countable $Y$, and similarly for $[B]_{\mu_{X}}$.
\end{df}

By Fubini's Theorem, $[B]_{\mu_{X}}$ has support $Y$ if and only if $$\{ y \in {^{Y}}2 \mid \mu_{X \setminus Y}(B_{y}) \not\in \{0,1\}\} \in \mathcal{N}_{Y}.$$




\begin{remark} Baire sets are built in countably many stages by countable unions and complements from sets of the form $C^{i}_{x}$, for $x \in X$ and $i \in 2$. It follows that they
have countable support, and can be coded by countable subsets of $X\times \omega$.
These codes give rise to reinterpretations of these sets in generic extensions. For infinite cardinals $\gamma$, we fix a (suppressed) coding
of Baire subsets of ${^{\gamma}}2$ by elements of ${^{\gamma}}2$, and let $eval(c)$ denote the Baire subset of ${^{\gamma}}2$ coded by $c$.
We refer the reader to \cite{Ku} for more on codes. Whenever possible, we prefer to talk about Baire and Borel sets in the ground model and their reinterpretations in generic extensions, and suppress mention of codes and their evaluations.

Given a code $c$ for a Baire subset of ${^{X}}2$, if the set coded by $c$ is empty, this remains true in any forcing extension.
(For each code $c$ for a Borel sets of reals, one can build a tree on $\omega \times \omega$ whose projection is (absolutely) equal to $eval(c)$, working recursively by Borel rank.
If the projection of such a tree is empty, this is witnessed absolutely by a ranking function. The corresponding fact for Baire sets follows,
as they have countable support.)
It follows that relations such as $x \in B$, $x \not\in B$ and $B \in \mathcal{N}_{X}$ (for a given Baire set $B$ in the ground model) are preserved when one passes from a Baire set to its reinterpretation in a
generic extension.

Similarly, a Polish space in an inner model naturally reinterprets in an outer model, via any countable dense subset. We typically ignore the distinction.
\end{remark}

Using the fact that $\mathbb{B}(X)$ is c.c.c., and arguing by induction on the rank of the Baire subsets of ${^{X}}2$, it is not hard to see that the generic function $F$ is a member of every member (reinterpreted) of the generic filter. Moreover, a function $F \colon X \to 2$ is $\mathbb{B}(X)$-generic over some model of ZF if and only if it is a member of every (reinterpreted) Baire set of $\mu_{X}$-measure 1 in the model.

Since $\mu_{X}$ is a finite countably additive measure, $\mathbb{B}(X)$ has no uncountable antichains. It follows that each $\mathbb{B}(X)$-name $\undertilde{\eta}$ for a real is decided by the restriction of the generic filter to $\mathbb{B}(Y)$, for some countable $Y \subseteq X$ depending only on the name itself, in which case we say that $\undertilde{\eta}$ has \emph{support} $Y$ (again, the support of $\undertilde{\eta}$ is not unique).




In order to study the set of universally measurable sets in the random algebra extension, we are going to look at those measures which are induced by names for reals. Let $P$ be a Polish space. For each $\mathbb{B}(X)$-name $\undertilde{\eta}$ for an element of $P$, there is a Borel probability measure $\sigma$ on $P$ defined by letting $\sigma(I)$ (for each Borel $I \subseteq \creals$) be $\mu_{X}(\truth{\undertilde{\eta} \in \check{I}})$, where $\truth{\undertilde{\eta} \in \check{I}}$ is the condition $p$ in $\mathbb{B}(X)$ with the property that $p \forces \undertilde{\eta} \in \check{I}$ and every condition incompatible with $p$ forces that $\undertilde{\eta} \not\in \check{I}$. To see that there is such a $p$, note that $\mathbb{B}(X)$ is c.c.c. and the collection of Baire subsets of ${^{X}}2$ is a $\sigma$-algebra.

The following is a standard fact about the random algebra (see Theorem 552P of \cite{F5} or Theorem 3.13 of \cite{Ku}).

\begin{thrm}\label{randfact} If $Y$ is a nonempty proper subset of $X$, then $\mathbb{B}(X)$ is forcing-equivalent to the iteration $\mathbb{B}(Y) * \mathbb{B}(X \setminus Y)$.
\end{thrm}


Each intermediate extension $V[H]$ then generates a new class of measures on $P$, those induced by the $\mathbb{B}(X \setminus Y)$-names for elements of $P$. Our first proof uses the measures $\sigma(H, \undertilde{\eta})$, where $\undertilde{\eta}$ is a $\mathbb{B}(X)$-name in $V$ for an element of a Polish space $P$ in $V$ and $H$ is a $V$-generic filter for $\mathbb{B}(Y)$, for some nonempty proper subset $Y$ of $X$. For each Borel set $I \subseteq P$, $\sigma(H, \undertilde{\eta})(I)$ is defined to be $$\mu_{X \setminus Y}(\truth{\undertilde{\eta}/H \in \check{I}}),$$ where, as before, $\truth{\undertilde{\eta}/H \in \check{I}}$ is the condition in $\mathbb{B}(X\setminus Y)$ (as interpreted in $V[H]$) forcing that $\undertilde{\eta}/H \in \check{I}$, such that every
condition incompatible with it forces that $\undertilde{\eta}/H \not\in \check{I}$ (here $\undertilde{\eta}/H$ is a $\mathbb{B}(X \setminus Y)$-name whose realization in the $\mathbb{B}(X \setminus Y)$-extension of $V[H]$ will be the same as the realization of $\undertilde{\eta}$ in the $\mathbb{B}(X)$-extension of $V$ induced by $H$ and the $V[H]$-generic filter for $\mathbb{B}(X \setminus Y)$).

The following fact appears in $\cite{Ku, F5}$. Combining it with Theorem \ref{first}, we have that consistently there exist
just continuum many universally measurable sets of reals.

\begin{thrm}\label{arith} Suppose that $\kappa$ is a cardinal such that $\kappa^{\mathfrak{c}} = \kappa$. Then $2^{\aleph_{0}} = 2^{(\mathfrak{c}^{V})} =
\kappa$ in the $\mathbb{B}(\kappa)$ extension.
\end{thrm}

\subsection{First proof}


\begin{thrm}\label{first} Let $P$ be a Polish space. For any nonempty set $X$, every universally measurable subset of $P$ is the union of $\mathfrak{c}^{V}$ many
Borel sets in the $\mathbb{B}(X)$ extension.
\end{thrm}

\begin{proof} Let $\underTilde{A}$ be a $\mathbb{B}(X)$-name for a universally measurable set of reals. Then for every $\mathbb{B}(X)$-name $\dot{m}$
for a  Borel measure on $P$ there are $\mathbb{B}(X)$-names $\underTilde{B}$ and $\underTilde{N}$ for Borel subsets of $P$ such that every condition forces
that $\underTilde{A} \bigtriangleup \underTilde{B} \subseteq \underTilde{N}$ and that $\underTilde{N}$ has $\dot{m}$-measure 0. Furthermore, every measure in the extension is the realization of a $\mathbb{B}(X)$-name with countable support, and $\underTilde{B}$ and $\undertilde{N}$ can be taken to have countable support. For any countable $Z \subseteq X$, there is a set of continuum many names giving rise to all the reals in the $\mathbb{B}(Z)$ extension. It follows then that there is a nonempty set $Y \subseteq X$ of cardinality at most $\mathfrak{c}$ such that if $H \subset \mathbb{B}(Y)$ is a $V$-generic filter, then for every Borel measure $\rho$ on $P$ in $V[H]$ there are Borel sets $B$, $N$ such that $\rho(N) = 0$ and every condition in $\mathbb{B}(X \setminus Y)$ forces that $\underTilde{A}/H \bigtriangleup \check{B} \subseteq \check{N}$ (this is somewhat abusive notation : $\check{B}$ and $\check{N}$ refer to the reinterpretations of $B$ and $N$ in the $\mathbb{B}(X \setminus Y)$-extension; strictly speaking we should fix codes $c$ and $d$ for $B$ and $N$ and refer to $eval(\check{c})$ and $eval(\check{d})$). For every $\mathbb{B}(X)$-name $\undertilde{\eta}$ in $V$, the measure $\sigma(H, \undertilde{\eta})$ exists in $V[H]$. Letting $B$, $N$ be the corresponding Borel sets for this measure, it follows that every condition in $\mathbb{B}(X \setminus Y)$ forces that $\undertilde{\eta}/H \not\in \check{N}$, and therefore that
$\undertilde{\eta}/H \in \underTilde{A}$ if and only if $\undertilde{\eta}/H \in \check{B}$. It follows that in the $\mathbb{B}(X)$ extension,
the realization of $\underTilde{A}$ is the union of all the (reinterpreted) Borel sets in the corresponding $\mathbb{B}(Y)$ extension (specifically, sets of the form $B \setminus N$) which are subsets of $\underTilde{A}$. Since each such Borel set is the realization of a name with support a countable subset of $Y$, there are just $\mathfrak{c}^{V}$ many such sets.
\end{proof}

The forcing axiom MA$_{\kappa}$ (Martin's Axiom for $\kappa$-many dense open sets) implies that the union of $\kappa$ many null sets is null, for any Borel measure on a Polish space (see the proof of Theorem 7.3 of \cite{Bl}). From this it follows (again, under MA$_{\kappa}$) that unions of $\kappa$ many Borel sets are universally measurable (since if $B_{\alpha}$ $(\alpha < \kappa)$ are $\mu$-measurable sets, the set
$B_{\alpha} \setminus \bigcup_{\beta < \alpha}B_{\beta}$ must be $\mu$-null for a tail of $\alpha$).
It is tempting then to think that consistently the universally measurable sets are exactly the unions of $\kappa$ many Borel sets, for some cardinal $\kappa$. The following fact rules this out, however.

\begin{thrm}[Grzegorek (see \cite{M84})] If $\lambda$ is the least cardinality of a nonmeasurable set of reals, then there is a universally null set of
cardinality $\lambda$.
\end{thrm}

A universally null set of cardinality $\lambda$ cannot be a union of less than $\lambda$ many Borel sets, since it would then have to contain a perfect set (we thank J\"{o}rg Brendle for pointing this out to us).

However, the argument given here shows the following fact.

\begin{thrm}\label{reverse} Let $P$ be a Polish space and let $X$ be a set of cardinality greater than $\mathfrak{c}$.
In the $\mathbb{B}(X)$ extension, for every set $A \subseteq P$, the following are equivalent:
\begin{enumerate}
\item $A$ is universally measurable,
\item $A$ and its complement are unions of less than $\mathfrak{c}$ many many Borel sets,
\item $A$ and its complement are unions of $\mathfrak{c}^{V}$ many Borel sets.
\end{enumerate}
\end{thrm}

Direction (1) $\Rightarrow$ (3) of Theorem \ref{reverse} is given by Theorem \ref{first}, plus the fact that a subset of
a Polish space is universally measurable if and only if its complement is. Direction (3) $\Rightarrow$ (2) follows from the assumption
that $|X| > \mathfrak{c}^{V}$.

David Fremlin has pointed out to us a relatively simple proof of the remaining direction of Theorem \ref{reverse}. By convention,
$cov(\mathcal{N})$ denotes the smallest cardinality of a collection of Lebesgue null sets whose union covers the real line. By Theorem \ref{isom},
for any finite Borel measure $\mu$ on any Polish space $P$, $cov(\mathcal{N})$ is the same as the smallest cardinality of a collection
of $\mu$-null sets covering $P$. It is standard fact that if $X$ is a set of cardinality greater than $\mathfrak{c}$, then $cov(\mathcal{N}) = \mathfrak{c}$ in the $\mathbb{B}(X)$-extension - this follows for instance from the fact that any $\mathbb{B}(\omega)$-generic real falls outside
of all (reinterpreted) ground model Lebesgue null sets. With these facts in mind, direction (2) $\Rightarrow$ (1) follows from the following argument, which is similar to Exercise 522Yk of \cite{F5}.

\begin{thrm} Suppose that $P$ is a Polish space, $\mu$ is an atomless finite Borel measure on $P$, and $A$ is a subset of $P$ such that
$A$ and $P \setminus A$ are unions of less than $cov(\mathcal{N})$ many $\mu$-measurable sets. Then $A$ is $\mu$-measurable.
\end{thrm}

\begin{proof} Fix $\kappa < cov(\mathcal{N})$ and $\mu$-measurable sets $B_{\alpha}$, $C_{\alpha}$ $(\alpha < \kappa)$ such that $A = \bigcup_{\alpha < \kappa} B_{\alpha}$ and $P \setminus A = \bigcup_{\alpha < \kappa} C_{\alpha}$. There exist countable subsets of $\kappa$, $E$ and $F$ such that
for all countable $G \subseteq \kappa$, $$\mu(\bigcup_{\alpha \in G}B_{\alpha}) \leq \mu(\bigcup_{\alpha \in E}B_{\alpha})$$ and
$$\mu(\bigcup_{\alpha \in G}C_{\alpha}) \leq \mu(\bigcup_{\alpha \in F}C_{\alpha}).$$ Let $H = \bigcup_{\alpha \in E}B_{\alpha} \cup \bigcup_{\alpha \in F}C_{\alpha}$ If $\mu(H) = \mu(P)$, then $A$ is $\mu$-measurable, and we are done. Otherwise, let $\mu'$ be the finite Borel measure on $P$ defined by letting $\mu'(I) = \mu(I \setminus H)$ for all Borel $I \subseteq P$. We have that for all $\alpha \in \kappa$, $\mu(B_{\alpha} \cup C_{\alpha} \cup H) = \mu(H)$, which implies that $\mu'(B_{\alpha} \cup C_{\alpha}) = 0$. Thus $P$ is a union of $\kappa$ many $\mu'$-null sets, giving a contradiction.
\end{proof}

In the rest of this section, we give our original proof of direction (2) $\Rightarrow$ (1) of Theorem \ref{reverse}.
This proof uses the following definition.

\begin{df} Suppose that $A$ is a subset of a Polish space $P$, $Q$ is a partial order,
and $G \subset Q$ is a $V$-generic filter.
The \emph{Borel reinterpretation} of $A$ in $V[G]$ is the union of all the reinterpreted ground model
Borel sets contained in $A$.
\end{df}

A version of the following fact is mentioned in \cite{FMW}.

\begin{lem}\label{fmwfact} Let $P$ be a Polish space, and let $A$ be a subset of $P$. Then the following
are equivalent.
\begin{enumerate}
\item $A$ is universally measurable.
\item The Borel reinterpretations of $A$ and $P \setminus A$ are complements in every $\mathbb{B}(\omega)$-extension.
\item For every infinite set $X$, the Borel reinterpretations of
$A$ and $P \setminus A$ are complements in every $\mathbb{B}(X)$-extension.
\end{enumerate}
\end{lem}

\begin{proof} The implication (3) $\Rightarrow$ (2) is immediate. For (1) $\Rightarrow$ (3), fix an infinite set $X$ and let  $G \subset \mathbb{B}(X)$ be a $V$-generic filter. The Borel reinterpretations of $A$ and $P \setminus A$ are clearly disjoint, so it suffices
to see that each $x \in (P)^{V[G]}$ is in one set or the other. Each such $x$ is the realization of a
$\mathbb{B}(X)$-name $\undertilde{\eta}$, and there is a corresponding measure $\sigma$ on $P$ defined by letting $\sigma(I)$
(for each Borel set $I \subseteq P$) be the $\mu_{X}$-measure of the condition in $\mathbb{B}(X)$ asserting that the realization of $\undertilde{\eta}$ is in $I$. Then there is a Borel set $B \subseteq P$ such that $A \bigtriangleup B$ is contained in a
$\sigma$-null Borel set $N$. Then $x \not\in N$, so $x$ is in the reinterpretation of either $B \setminus N$ or $(P \setminus B) \setminus N$.
Finally, $B \setminus N \subseteq A$ and $(P \setminus B) \setminus N \subseteq P \setminus A$.

For (2) $\Rightarrow$ (1), let $\nu$ be an atomless probability measure on $P$. Applying Theorem \ref{isom}, let $\pi \colon {^{\omega}}2 \to P$ be a Borel isomorphism such that $\nu(\pi[I]) = \mu_{\omega}(I)$ for every Borel $I \subseteq {^{\omega}}2$, where $\mu_{\omega}$ is as defined at the beginning of Section \ref{randmeas}. Letting $F$ denote the $\mathbb{B}(\omega)$-generic function, we have a dense set of conditions in $\mathbb{B}(\omega)$ forcing that $\pi(F)$ will be contained in a ground model Borel set which is either contained in $A$ or its complement. Since $\mathbb{B}(\omega)$ is c.c.c., there exist Borel sets $B_{0}$ and $B_{1}$ such that $\pi[B_{0}] \subseteq A$, $\pi[B_{1}] \subseteq (P \setminus A)$ and $\mu_{\omega}(B_{0} \cup B_{1}) = 1$. Then $\pi[B_{0}]$ and $\pi[P \setminus (B_{0} \cup B_{1})]$ witness universal measurability for $A$ and $\nu$.
\end{proof}

The fact appears in \cite{F5}, with a different proof.

\begin{thrm}\label{univrei} Suppose that $A$ is a universally measurable subset of a Polish space $P$, $X$ is a nonempty set and
$F \colon X \to 2$ is a $V$-generic function for $\mathbb{B}(X)$. Then the Borel reinterpretation of $A$ in $V[F]$ is universally
measurable.
\end{thrm}

\begin{proof}[First proof of Theorem \ref{univrei}.] Let $A^{*}$ denote the Borel reinterpretation of $A$ in $V[F]$. If $A^{*}$ is not universally measurable in $V[F]$, then
there is a generic extension of $V[F]$ by $\mathbb{B}(\omega)$ (as understood in $V[F]$) in which the Borel reinterpretations of $A^{*}$ and its complement (using Borel sets from $V[F]$) are not complements. By Theorem \ref{randfact}, this extension is a generic extension of $V$ by the forcing $\mathbb{B}((X \times \{0\}) \cup (\omega \times \{1\}))$ in which the Borel reinterpretations of $A$ and its complement (using Borel sets from $V$)
are not complements. This contradicts the universal measurability of $A$ in $V$, by Lemma \ref{fmwfact}.
\end{proof}

We give a second proof which uses material from Section \ref{moresec}.

\begin{proof}[Second proof of Theorem \ref{univrei}.]  By Theorem \ref{isom}, it suffices to prove the theorem in the case where $P$ is $\creals$. A measure on $\creals$ in $V[F]$ is coded by an element of $\creals$,
which in turn is realized by a Borel function $f \colon {^{Y}}2 \to \creals$ for some countable $Y \subseteq X$,
as in Theorem \ref{namemeas} below. Letting $m_{x}$ denote the measure coded by $x \in \creals$, $f$ gives rise to a measure $\upsilon$ on $\creals$: $\upsilon(E) = \int m_{f(y)}(E)\, d\pi$, where $\pi$ is the product measure on ${^{Y}}2$. Since $A$ is universally measurable, there is a
Borel set $B \subseteq \creals$ such that $A \bigtriangleup B$ is contained in an $\upsilon$-null set $N$. Then, letting $F$ be the generic function, $F \restrict Y$ is in the set of $y \in {^{Y}}2$ such that $N$ is $m_{f(y)}$-null, so the reinterpretation of $N$ is $m_{f(F\restrict Y)}$-null. Since $B \setminus N$ is contained in the Borel reinterpretation of
$A$, and $(\creals \setminus B) \setminus N$ is contained in the Borel reinterpretation of $(\creals \setminus A)$, the symmetric difference of the Borel reinterpretation of $A$ with the reinterpretation of $B$ is contained in the reinterpretation of $N$.
\end{proof}

\begin{proof}[Proof of direction (2) $\Rightarrow$ (1) of Theorem \ref{reverse}.] Let $F$ be a $\mathbb{B}(X)$-generic function, and let $\kappa$
be a cardinal less than $|X|$. Suppose that $A$ is a set of reals in $V[F]$,
and that $\bar{B} = \{ B_{\alpha} : \alpha < \kappa\}$ and $\bar{C} = \{ C_{\alpha} : \alpha < \kappa \}$ are collections of Borel sets in $V[F]$ such that $A = \bigcup_{\alpha < \kappa}B_{\alpha}$ and $P \setminus A = \bigcup_{\alpha < \kappa} C_{\alpha}$. Then there is a set $Y \subset X$ such that $X \setminus Y$ is uncountable, and such that $\bar{B}$ and $\bar{C}$ (that is, the corresponding sets of codes) are in $V[F \restrict Y]$. By Theorem \ref{univrei}, it suffices to see that $\bigcup \bar{B}$ is universally measurable in $V[F \restrict Y]$.

Suppose towards a contradiction that that $\bigcup \bar{B}$ is not universally measurable in $V[F \restrict Y]$.
Then there is a condition $p$ in $\mathbb{B}(\omega)$ forcing that some new member of $P$ will not be in either $\bigcup \bar{B}$ or $\bigcup \bar{C}$. Since $X \setminus Y$ is uncountable, there are densely many such conditions in $\mathbb{B}(X \setminus Y)$. To see this, let $D \subseteq {^{\omega}}2$ be a Borel set in $V[F \restrict Y]$ such that $p = [D]_{\mu_{\omega}}$ and let $E \subseteq {^{X \setminus Y}}2$ be a non-$\mu_{X \setminus Y}$-null Baire set in $V[F \restrict Y]$. Then $E$ has countable support, so we may fix an injection $i \colon \omega \to X \setminus Y$ such that the range of $i$ is disjoint from the support of $E$. Let $E'$ be the set of $f \in E$ such that for some $h \in D$, $f(i(n)) = h(n)$ for all $n \in \omega$. Then $[E']_{\mu_{X \setminus Y}}$ forces that $F^{*} \circ i$ (where $F^{*}$ is the $\mathbb{B}(X \setminus Y)$-generic function) will be $V[F \restrict Y]$-generic for the restriction of the partial order $\mathbb{B}(\omega)$ below $p$, and therefore that $\bigcup \bar{B}$ and $\bigcup \bar{C}$ will not be complements. Since $\bigcup \bar{B}$ and $\bigcup \bar{C}$ are complements in $V[F]$, we have the desired contradiction.
\end{proof}

\subsection{More on the random algebra and measures}\label{moresec}

Our original solution to Mauldin's question requires more background information on the random algebra.
Again, we fix a nonempty set $X$.

\begin{df}\label{updown} Given a Baire set $B \subseteq {^{X}}2$ and a nonempty set $Y \subset X$, $B_{Y}$ denotes the set of $y \in {^{Y}}2$ such that
$B_{y} \not\in \mathcal{N}_{X \setminus Y}$. Given a Baire set $E \subseteq {^{Y}}2$, $E^{X}$ is the set of $x \in {^{X}}2$ such that $x \restrict Y \in E$.
\end{df}

\begin{remark}\label{suppind} If $B$ is a Baire subset of ${^{X}}2$ and $B$ has support $Y \subseteq X$, then $\mu_{X}(B) = \mu_{Y}(B_{Y})$.
\end{remark}

\begin{notation} By Fubini's Theorem, if $B$, $B'$ are Baire subsets of ${^{X}}2$, $Y$ is a nonempty proper subset of $X$ and $B \bigtriangleup B' \in \mathcal{N}_{X}$, then $B_{Y}\bigtriangleup B'_{Y} \in \mathcal{N}_{Y}$, so we can let $([B]_{\mu_{X}})_{Y}$ denote $[B_{Y}]_{\mu_{Y}}$. Similarly, if $E_{0}$, $E_{1}$ are Baire subsets of ${^{Y}}2$ such that $E_{0} \bigtriangleup E_{1} \in \mathcal{N}_{Y}$, $E_{0}^{X} \bigtriangleup E_{1}^{X}\in \mathcal{N}_{X}$, so
we can let $[E_{0}]_{\mu_{Y}}^{X}$ denote $[E_{0}^{X}]_{\mu_{X}}$.
\end{notation}

\begin{remark}\label{filtfact} As mentioned above, if $G \subset \mathbb{B}(X)$ is a $V$-generic filter, $F$ is the corresponding generic function, and $B \subseteq {^{X}}2$ is a Baire set in $V$, then $F$ is an element of $B$ (reinterpreted in the extension) if and only if $[B]_{\mu_{X}} \in G$. This can be proved by induction on the rank of $B$ (i.e., the number of steps needed to generate $B$ from sets of the form $C^{i}_{x}$).

Furthermore, if $Z \subset Y$ are nonempty subsets of $X$, and $B \subseteq {^{X}}2$ has support $Y$, then, by Fubini's theorem (applied twice), the set of $x \in B$ such that $(B_{Y})_{x \restrict Z} \in \mathcal{N}_{Y \setminus Z}$ has measure 0 (to see this, first note that
$(B_{Y})_{z} \in \mathcal{N}_{Y \setminus Z}$ if and only if $B_{z} \in \mathcal{N}_{X \setminus Z}$).
It follows that if $[B]_{\mu_{X}} \in G$ then
$(B_{Y})_{F \restrict Z} \not\in \mathcal{N}_{Y \setminus Z}$
in $V[G]$, again using the reinterpretation of $B$.
\end{remark}



Given a Polish space $P$ in the ground model, every $\mathbb{B}(X)$-name for an element of $P$ is represented by a Borel function in the following way (see Definition 551C of \cite{F5}, or Theorems 3.1.5 and 3.1.7 of \cite{BaJu}, and their proofs).

\begin{thrm}\label{namemeas} Suppose that $X$ is a nonempty set, $P$ is a Polish space and $\undertilde{\eta}$ is a $\mathbb{B}(X)$-name for an element of $P$. Suppose that $Y \subseteq X$ is countable and $\undertilde{\eta}$ has support $Y$. Then there exists a Borel function $f \colon {^{Y}}2 \to P$ such that $[{^{X}}2]_{\mu_{X}}$ forces in $\mathbb{B}(X)$ that $f(F \restrict Y) = \undertilde{\eta}_{G}$, where $G \subset \mathbb{B}(X)$ is the generic filter, $F \colon X \to 2$ is the associated generic function and $f$ is identified with its reinterpretation in the generic extension.
\end{thrm}

\begin{remark}\label{nofun} If $f \colon {^{Y}}2 \to P$ and $f' \colon {^{Y'}}2 \to P$ are two functions as in Theorem \ref{namemeas} for the same
$X$ and $\undertilde{\eta}$, then $\mu_{X}(\{ x \in {^{X}}2 \mid f(x \restrict Y) = f'(x \restrict Y')\}) = 1$.
\end{remark}


Theorem \ref{namemeas} gives us another way to associate measures to names for reals.

If
\begin{itemize}
\item $B$ is a Baire member of $\mathcal{P}({^{X}}2) \setminus \mathcal{N}_{X}$,
\item $Y \subset X$ is countable and $B$ has support $Y$,
\item $P$ is a Polish space,
\item $f \colon {^{Y}}2 \to P$ is a Borel function,
\item $Z$ is a proper subset of $Y$,
\item $z \in {^{Z}}2$ and $(B_{Y})_{z} \not \in \mathcal{N}_{Y \setminus Z}$,
\end{itemize}
then there is a Borel probability measure $\rho(z, B, f)$ on $P$ defined by letting $\rho(z, B, f)(I)$ be $$\dfrac{\mu_{Y \setminus Z}((B_{Y})_{z} \cap \{ y \in {^{Y \setminus Z}}2 \mid f(y \cup z) \in I\})}{\mu_{Y \setminus Z}((B_{Y})_{z})}$$
for all Borel $I \subseteq P$. (If we removed the denominator of this expression we would still have a suitable measure, just not
necessarily a probability measure.)

It follows from Remark \ref{filtfact} that if
\begin{itemize}
\item $G \subset \mathbb{B}(X)$ is a $V$-generic filter,
\item $F$ is the associated generic function,
\item $p = [B]_{\mu_{X}}$ is a member of $\mathbb{B}(X)$,
\item $\undertilde{\eta}$ is a $\mathbb{B}(X)$-name,
\item $Y \subset X$ is countable and nonempty, and $B$ and $\undertilde{\eta}$ have support $Y$,
\item $f \colon {^{Y}}2 \to P$ witnesses Theorem \ref{namemeas} for $\undertilde{\eta}$,
\item $Z$ is a proper subset of $Y$,
\item $p_{Z}^{X} \in G$, (here we are using the notation specified after Remark \ref{suppind})
\end{itemize}
then in $V[G]$ (indeed, in $V[F \restrict Z]$) there is a Borel probability measure $\nu(F \restrict Z, p, \undertilde{\eta})$ on $P$ defined by letting $$\nu(F \restrict Z, p, \undertilde{\eta})(I) = \rho(F \restrict Z, B, f)(I)$$ for all Borel $I \subseteq \creals$, using the reinterpretations of $f$ and $B$ in $V[G]$. By Remarks \ref{suppind}, \ref{filtfact} and \ref{nofun}, $\nu(F \restrict Z, p, \undertilde{\eta})$ does not depend on the choice of $f$, $Y$ or $B$.





The reader may check that the measure $\nu(F \restrict Z, [{^{X}}2]_{\mu_X}, \undertilde{\eta})$ is the same as the measure $\sigma(H, \undertilde{\eta})$ from above, when $F \restrict Z$ is the part of the generic function corresponding to $H$. We will not use this fact, however. The following lemma is trivial when we use the measures $\sigma(H, \undertilde{\eta})$.

\begin{lem}\label{outnull} Suppose that
\begin{itemize}
\item $X$ is a nonempty set,
\item $G \subset \mathbb{B}(X)$ is a $V$-generic filter,
\item $F$ is the associated generic function,
\item $p$ is a member of $\mathbb{B}(X)$,
\item $P$ is a Polish space,
\item $\undertilde{\eta}$ is a $\mathbb{B}(X)$-name for a member of $P$,
\item $Y \subset X$ is countable and nonempty, and $p$ and $\undertilde{\eta}$ have support $Y$,

\item $Z$ is a proper subset of $Y$,
\item $p_{Z}^{X} \in G$,
\item $\undertilde{\xi}$ is a $\mathbb{B}(X)$-name for an element of $\creals$ coding a $\nu(F \restrict Z, p, \undertilde{\eta})$-null set,
\item $W \subseteq X\setminus (Y \setminus Z)$ is countable,
\item $\undertilde{\xi}$ has support $W$.
\end{itemize}
Then $\undertilde{\eta}_{G} \not\in eval(\undertilde{\xi}_{G})$.
\end{lem}

\begin{proof} Let $f \colon {^{Y}}2 \to P$ witness Theorem \ref{namemeas} for $\undertilde{\eta}$, and
let $B$ be a Baire subset of ${^{X}}2$ in $V$ with support $Y$ such that $p = [B]_{\mu_{X}}$. Let $B_{0}$ be the set of $x \in B$ such that $(B_{Y})_{x \restrict Z} \not\in \mathcal{N}_{Y \setminus Z}$.
By Remark \ref{filtfact}, $B_{0} \bigtriangleup B \in \mathcal{N}_{X}$. Let $g \colon {^{W}}2 \to \creals$ witness Theorem \ref{namemeas} for
$\undertilde{\xi}$.
Then the set of $x \in B_{0}$ such that $g(x \restrict W)$ is a code for a $\rho(x \restrict Z, B_{0}, f)$-null set is in $G$. Call this set $B_{1}$. For each $x \in B_{1}$, the set of $y \in {^{Y \setminus Z}}2$ such that
$$y \cup (x\restrict Z) \in (B_{1})_{Y}$$ and
$$f(y \cup (x \restrict Z)) \text{ is in the set coded by }g(x \restrict W)$$
is in $\mathcal{N}_{Y \setminus Z}$.
Let $B_{2}$ be the set of $x \in B_{1}$ such that $f(x \restrict Y)$ is not in the set coded by $g(x \restrict W)$. Then $B_{1} \bigtriangleup B_{2} \in \mathcal{N}_{X}$, so $[B_{1}]_{\mu_{X}} = [B_{2}]_{\mu_{X}} \in G$.
\end{proof}

Given $X \neq \emptyset$, $Y \subseteq X$ and an injective function $\pi \colon Y \to X$, we extend $\pi$ to all Baire subsets of ${^{X}}2$
with support $Y$, by letting $\pi(C^{i}_{x}) = C^{i}_{\pi(x)}$ for all $y \in Y$ and $i \in 2$, and letting $\pi(\bigcap_{i \in \omega}B_{i}) = \bigcap_{i \in \omega}\pi(B_{i})$ and $\pi({^{X}}2 \setminus B) =
{^{X}}2 \setminus \pi(B)$. This map preserves $\mu_{X}$, so it extends also to conditions in $\mathbb{B}(X)$ with support $Y$. We extend $\pi$
to $\mathbb{B}(X)$-names with support $Y$ for members of ground model Polish spaces as follows. For any Polish space $P$ and any $\mathbb{B}(X)$-name $\undertilde{\eta}$ with support $Y$ for a member of $P$,  $\pi(\undertilde{\eta})$ is a $\mathbb{B}(X)$-name with support $\pi[Y]$ for a member of $P$. Furthermore, if $\undertilde{\theta}$ is a $\mathbb{B}(X)$-name with support $Y$ for a Borel subset of $P$ (i.e., a name for a code for the realization of $\undertilde{\theta}$ has support $Y$), then $\pi(\undertilde{\theta})$ is a $\mathbb{B}(X)$-name with support $\pi[Y]$ for a Borel subset of $P$ (the name induced by the $\pi$-image of the name for the code corresponding to $\undertilde{\theta}$), and, for each $p \in \mathbb{B}(X)$ with support $Y$, $\pi(p) \forces \pi(\undertilde{\eta}) \in \pi(\undertilde{\theta})$ if and only if $p \forces \undertilde{\eta} \in \undertilde{\theta}$.




The following lemma is immediate from the definitions.

\begin{lem}\label{goodmeas} Suppose that
\begin{itemize}
\item $X$ is a nonempty set,
\item $F \colon X \to 2$ is a $V$-generic function for $\mathbb{B}(X)$,
\item $Y$ is a countable subset of $X$,
\item $Z$ is a proper subset of $Y$,
\item $\pi \colon Y \to X$ is an injection which fixes the members of $Z$,
\item $p = [B]_{\mu_{X}}$ is a condition in $\mathbb{B}(X)$ with support $Y$,
\item $p_{Z}^{X} \in G$,
\item $P$ is a Polish space,
\item $\undertilde{\eta}$ is a $\mathbb{B}(X)$-name for an element of $P$,
\item $p$ and $\undertilde{\eta}$ have support $Y$,
\end{itemize}
Then in $V[F]$, $\nu(F \restrict Z, \pi(p), \pi(\undertilde{\eta})) = \nu(F \restrict Z, p, \undertilde{\eta})$.
\end{lem}

\subsection{Second proof}

In this section we present our original proof of our main theorem. We note that David Fremlin \cite{F09} has written up the argument of this section in the style of his treatise on measure theory (e.g. \cite{F2, F4, F5}).

In this section we force with $\mathbb{B}(\kappa)$, where $\kappa^{\mathfrak{c}} = \kappa$.
So the $X$'s from the previous section will become $\kappa$'s. Likewise, $Y$, $Z$ and $W$ will become $b$, $a$ and $d$.

\begin{thrm}\label{mainthrm} Suppose that $\kappa$ is a cardinal such that $\kappa^{\mathfrak{c}} = \kappa$, and let $P$ be a Polish space.
Then $\mathbb{B}(\kappa)$ forces that
$2^{\aleph_{0}} = 2^{\aleph_{1}} = \kappa$ and that the set of universally measurable subsets of $P$ has cardinality $\kappa$.
\end{thrm}

\begin{proof}
It follows from Theorem \ref{arith} that $\mathbb{B}(\kappa)$ forces $2^{\aleph_{0}} = 2^{\aleph_{1}} = \kappa$.
It remains to see that the set of universally measurable subsets of $P$ has cardinality $\kappa$.


Let $T$ be the set of $\langle a, b, p, \undertilde{\eta}\rangle$ such that
\begin{itemize}
\item $p \in \mathbb{B}(\kappa)$,
\item $b \in [\kappa]^{\aleph_{0}}$,
\item $a\subseteq b$,
\item $\undertilde{\eta}$ is a $\mathbb{B}(\kappa)$-name for an element of $P$,
\item $p$ and $\undertilde{\eta}$ have support $b$.
\end{itemize}

We will associate to each $\mathbb{B}(\kappa)$-name $\underTilde{A}$ for a universally measurable subset of $P$ a set $u(\underTilde{A}) \subset \kappa$ of cardinality $2^{\aleph_{0}}$
and a continuum-sized collection $R(\underTilde{A})$ of sequences $\langle t, a, b, p, \undertilde{\eta}\rangle$ such that
\begin{itemize}
\item $t \in \{ True, False\}$,
\item $\langle a, b, p, \undertilde{\eta}\rangle \in T$,
\end{itemize}
in such a way that the realization of $\underTilde{A}$ is completely determined by $u(\underTilde{A})$, $R(\underTilde{A})$ and the generic filter (so $\underTilde{A}$ itself is not needed). Since there are only $\kappa$ many such pairs $(u(\underTilde{A})$, $R(\underTilde{A}))$ in the ground model, the collection of universally measurable subsets of $P$ in the extension will have cardinality $\kappa$.

We say that two elements $\langle a, b, p, \undertilde{\eta}\rangle$, $\langle a', b', p', \undertilde{\eta}'\rangle$ of $T$ are \emph{isomorphic}
if $a = a'$ and there is a bijection $\pi \colon b \to b'$ which fixes the members of $a$ such that $p' = \pi(p)$ and $\undertilde{\eta}' = \pi(\undertilde{\eta})$.
The following claim follows from the fact that every $\langle a, b, p, \undertilde{\eta} \rangle \in T$ is isomorphic to a sequence $\langle a, b', p', \undertilde{\eta}' \rangle \in T$ with $b' \subseteq a \cup (\sup(a), \sup(a) + \omega)$.

\begin{claim}[1]For each countable $a^{*} \subset \kappa$, there are $2^{\aleph_{0}}$ many isomorphism classes of $\{ \langle a,b, p, \undertilde{\eta}\rangle \in T \mid a = a^{*}\}$.
\end{claim}

Let $\underTilde{A}$ be a $\mathbb{B}(\kappa)$-name for a subset of $P$ such that every condition forces that the realization of $\underTilde{A}$ will be universally measurable. Let $T^{+}$ be the set of $\langle t, a, b, p, \undertilde{\eta}\rangle$ such that
\begin{itemize}
\item $\langle a, b, p, \undertilde{\eta}\rangle \in T$
\item $t \in \{ True, False\}$,
\item $p \forces ``\undertilde{\eta} \in \underTilde{A} \Leftrightarrow t$."
\end{itemize}
We say that two elements $\langle t, a, b, p, \undertilde{\eta}\rangle$, $\langle t, 'a', b', p', \undertilde{\eta}'\rangle$ of $T^{+}$ are \emph{isomorphic} if $\langle a, b, p, \undertilde{\eta}\rangle$ and $\langle a', b', p', \undertilde{\eta}'\rangle$ are isomorphic and $t = t'$.


Suppose that $C$ is an isomorphism class of $T$ such that $a$ is a proper subset of $b$ for every $\langle a, b, p, \undertilde{\eta}\rangle$ in $C$.
For each $\langle a, b, p, \undertilde{\eta}\rangle \in C$, the condition $p_{a}^{\kappa}$ is the same. By Lemma \ref{goodmeas}, $p_{a}^{\kappa}$ forces that the measure $\nu(F \restrict a, p, \undertilde{\eta})$ exists and is the same for all $\langle a, b, p, \undertilde{\eta}\rangle \in C$, where $F$ is the generic function. Let $\undertilde{\nu}$ be a $\mathbb{B}(\kappa)$-name for this measure, as forced by $p_{a}^{\kappa}$. Since $\underTilde{A}$ is a $\mathbb{B}(\kappa)$-name for a universally measurable set, there exist $\mathbb{B}(\kappa)$-names $\undertilde{\theta}$ and $\undertilde{\zeta}$ for Borel subsets of $P$ such that $p_{a}^{\kappa}$ forces that $\undertilde{\zeta}_{G}$ will be $\undertilde{\nu}_{G}$-null, and also that the symmetric difference of $\underTilde{A}_{G}$ and $\undertilde{\theta}_{G}$ will be contained in $\undertilde{\zeta}_{G}$, where $G$ is the generic filter.
Let $D(C)$ be the set of all countable $d \subset \kappa$ for which there exist $\mathbb{B}(\kappa)$-names for codes for such $\undertilde{\theta}_{G}$ and $\undertilde{\zeta}_{G}$ with support $d$.
For notational convenience, let $D(C) = \{ \emptyset \}$ whenever $C$ is an isomorphism class of $T$ such that $a = b$ for all $\langle a, b, p, \undertilde{\eta} \rangle$ in $C$.

For each isomorphism class $C^{*}$ of $T^{+}$, let $K(C^{*})$ be the set of countable $c \subset \kappa$ such that $(b \setminus a) \cap c$ is nonempty for all $\langle t, a, b, p, \undertilde{\eta} \rangle \in C^{*}$.

\begin{claim}[2] Suppose that
\begin{itemize}
\item $C$ is an isomorphism class of $T$,
\item $d \in D(C)$,
\item $\langle a, b^{0}, p^{0}, \undertilde{\eta}^{0}\rangle$ and $\langle a, b^{1}, p^{1}, \undertilde{\eta}^{1}\rangle$ are both in $C$,
\item $\langle t, a, b^{0}, p^{0}, \undertilde{\eta}^{0}\rangle \in T^{+}$,
\item $K(C^{*}) = \emptyset$, where $C^{*}$ is the isomorphism class of $\langle t, a, b^{0}, p^{0}, \undertilde{\eta}^{0}\rangle$,
\item $(b^{1} \setminus a) \cap d = \emptyset$.
\end{itemize}
Then $\langle t, a, b^{1}, p^{1}, \undertilde{\eta}^{1}\rangle \in T^{+}$.
\end{claim}

Proof of Claim 2: If $a = b$ for all $\langle a, b, p, \undertilde{\eta}\rangle \in C$, then the claim follows immediately,
so suppose that $a$ is a proper subset of $b$ for all $\langle a, b, p, \undertilde{\eta} \rangle \in C$.
Applying the fact that $K(C^{*}) = \emptyset$,
we may assume, by replacing $\langle t^{0}, a, b^{0}, p^{0}, \undertilde{\eta}^{0}\rangle$ with an isomorphic copy if necessary, that $b^{0} \setminus a$ and $b^{1} \setminus a$ are disjoint from $d$ and each other.

Since $d \in D(C)$, there are $\mathbb{B}(\kappa)$-names with support $d$ for members of $\creals$ inducing $\mathbb{B}(\kappa)$-names $\undertilde{\theta}$ and $\undertilde{\zeta}$ for Borel subsets of $P$ such that
$p_{a}^{\kappa}$ forces that $\undertilde{\zeta}_{G}$ will be $\undertilde{\nu}_{G}$-null, and also that the symmetric difference of $\underTilde{A}_{G}$ and $\undertilde{\theta}_{G}$ will be contained in $\undertilde{\zeta}_{G}$, where $G$ is the generic filter.
By Lemma \ref{outnull}, $p_{a}^{\kappa}$ forces that the realizations of $\undertilde{\eta}^{0}$ and $\undertilde{\eta}^{1}$ will both fall outside of the realization of $\undertilde{\zeta}$.

 Let $\pi$ be a permutation of $\kappa$, fixing the members of $a \cup d$, such that $p^{1} = \pi(p^{0})$ and $\undertilde{\eta}^{1} = \pi(\undertilde{\eta}^{0})$. Then $\theta = \pi(\undertilde{\theta})$, so we have the following.

\begin{tabular}{ r c l r}
  \(t = True\) & \(\Longleftrightarrow\) & \(p^{0} \forces \undertilde{\eta}^{0} \in \underTilde{A}\) & (by the definition of $T^{+}$)\\
   & \(\Longleftrightarrow\) & \(p^{0} \forces \undertilde{\eta}^{0} \in \undertilde{\theta}\) & (by the previous paragraph)\\
   & \(\Longleftrightarrow\) & \(\pi(p^{0}) \forces \pi(\undertilde{\eta}^{0}) \in \pi(\undertilde{\theta})\) & (by the remarks before Lemma \ref{goodmeas}) \\
   & \(\Longleftrightarrow\) & \(p^{1} \forces \undertilde{\eta}^{1} \in \undertilde{\theta}\) & (since $p^{1} = \pi(p^{0})$, $\undertilde{\eta}^{1} = \pi(\undertilde{\eta}^{0})$ and $\theta = \pi(\undertilde{\theta})$)\\
   & \(\Longleftrightarrow\) & \(p^{1} \forces \undertilde{\eta}^{1} \in \underTilde{A}\) & (by the previous paragraph)
\end{tabular}

\vspace{\baselineskip}
\noindent Similarly, $t = False$ if and only if $p^{1} \forces \undertilde{\eta}^{1} \not\in \underTilde{A}$. Thus $\langle t, a, b^{1}, p^{1}, \undertilde{\eta}^{1} \rangle \in T^{+}$. This completes the proof of Claim 2.

\vspace{\baselineskip}




For each isomorphism class $C$ of $T$, fix a set $d(C) \in D(C)$. For each isomorphism class $C^{*}$ of $T^{+}$, fix a set $k(C^{*}) \in K(C^{*}) \cup \{\emptyset\}$ such that $k(C^{*}) \in K(C^{*})$ if $K(C^{*}) \neq \emptyset$. Applying Claim 1, let $u(\underTilde{A})$ be a subset of $\kappa$ of cardinality $2^{\aleph_{0}}$ such that
\begin{itemize}
\item $d(C) \subset u(\underTilde{A})$ whenever $C$ is the isomorphism class of a sequence $\langle a, b, p, \undertilde{\eta} \rangle$
with $a \subset u(\underTilde{A})$,
\item $k(C^{*}) \subset u(\underTilde{A})$ whenever $C^{*}$
is the isomorphism class of a sequence
$$\langle t, a, b, p, \undertilde{\eta}\rangle \in T^{+}$$ with $a \subset u(\underTilde{A})$.
\end{itemize}
Let $T(u(\underTilde{A}))$ be the set of $\langle a, b, p, \undertilde{\eta}\rangle \in T$ such that $a = b \cap u(\underTilde{A})$, and define $T^{+}(u(\underTilde{A}))$ similarly. Then for any $\langle t, a, b, p, \undertilde{\eta} \rangle \in T^{+}(u(\underTilde{A}))$, $K(C^{*}) = \emptyset$, where $C^{*}$ is the isomorphism class of $\langle t, a, b, p, \undertilde{\eta}\rangle$ (since $k(C^{*}) \subset u(\underTilde{A})$).

Let $b(\underTilde{A}) = (sup(u(\underTilde{A})), sup(u(\underTilde{A}))+\omega)$.
Let $R(\underTilde{A})$ be the set of sequences $\langle t, a, b, p, \undertilde{\eta}\rangle \in T^{+}(u(\underTilde{A}))$ such that $b \setminus a = b(\underTilde{A})$.


Let $G \subset \mathbb{B}(\kappa)$ be a $V$-generic filter. We claim that $\underTilde{A}_{G}$ can be recovered from $u(\undertilde{A})$, $R(\underTilde{A})$ and $G$. To see this, let $\undertilde{\eta}$ be a $\mathbb{B}(\kappa)$-name for an element of
$P$. Let $b \subset \kappa$ be countable such that $\undertilde{\eta}$ has support $b$ and $b \setminus u(\underTilde{A})$ is infinite, and let $a = b \cap u(\underTilde{A})$. Let $\pi \colon b \to a \cup b(\underTilde{A})$ be a bijection fixing $a$. Then if $\undertilde{\eta}_{G} \in \underTilde{A}_{G}$, there exists $p^{0} \in G$ such that
\begin{itemize}
\item $\langle a, b, p^{0}, \undertilde{\eta}\rangle \in T(u(\underTilde{A}))$,
\item $\langle True, a, b(\underTilde{A}), \pi(p^{0}), \pi(\undertilde{\eta})\rangle \in R(\underTilde{A})$ (by Claim 2).
\end{itemize}
Furthermore, if $\undertilde{\eta}_{G} \not\in \underTilde{A}_{G}$, there exists $p^{1}\in G$ such that
\begin{itemize}
\item $\langle a, b, p^{1}, \undertilde{\eta}\rangle \in T(u(\underTilde{A}))$,
\item $\langle False, a, b(\underTilde{A}), \pi(p^{1}), \pi(\undertilde{\eta})\rangle \in R(\underTilde{A})$ (by Claim 2).
\end{itemize}
However, such conditions $p^{0}$ and $p^{1}$ cannot both exist, since they would be compatible, and thus $\pi(p^{0})$ and $\pi(p^{1})$ would be compatible.
\end{proof}

\section{CPA}\label{CPA}

The consistency of the negative answer to Mauldin's question follows from a very slight variation of Ciesielski and Pawlikowski's proof \cite{CiPa} that every universally null set has cardinality at most $\aleph_{1}$ in the iterated Sacks model. The modified version of their proof shows that every universally measurable set is the union of at most $\aleph_{1}$ many perfect sets and singletons in this model. Following \cite{CiPa}, we say that a \emph{cube} is a continuous injection from $\Pi_{n \in \omega}C_{n}$ to $P$, where $P$ is a Polish space and each $C_{n}$ is a
perfect subset of $\creals$. Let $\mathcal{F}_{cube}$ denote the set of cubes, and let $Perf(P)$ denote the collection of perfect subsets of $P$. A set $\mathcal{E} \subseteq Perf(P)$ is said to be $\mathcal{F}_{cube}$-\emph{dense} if for each $f \in \mathcal{F}_{cube}$ there is a $g \in \mathcal{F}_{cube}$ such that $g \subseteq f$ and $range(g) \in \mathcal{E}$. The axiom CPA$_{cube}(P)$ says that for every $\mathcal{F}_{cube}$-dense $\mathcal{E} \subseteq Perf(P)$ there is a $\mathcal{E}_{0} \subseteq \mathcal{E}$ such that $|\mathcal{E}_{0}| \leq \aleph_{1}$ and $|\creals \setminus \bigcup\mathcal{E}_{0}| \leq \aleph_{1}$.

The two following lemmas are Fact 1.0.2 and Claim 1.1.4 of \cite{CiPa}.

\begin{lem} A set $\mathcal{E} \subseteq Perf(P)$ is $\mathcal{F}_{cube}$-dense if and only if for every continuous injection $f \colon \Pi_{n \in \omega}\creals \to P$ there is a cube $g \subseteq f$ such that $range(g) \in \mathcal{E}$.
\end{lem}

\begin{lem}\label{secondquote} If $D$ is a Borel subset of $\Pi_{n \in \omega}\creals$ and $D$ has positive measure in the usual product measure on $\Pi_{n \in \omega}\creals$, then $D$ contains a
set of the form $\Pi_{n \in \omega}C_{n}$, where each $C_{n} \in Perf(\creals)$.
\end{lem}

The proof of the following is essentially the same as the proof of Theorem 1.1.4 of \cite{CiPa}, which shows that CPA$_{cube}(P)$ implies that universally null sets have cardinality less than or equal to $\aleph_{1}$.

\begin{thrm} If $P$ is a Polish space and CPA$_{cube}(P)$ holds, then every universally measurable set is the union of at most $\aleph_{1}$ many sets, each of which
is either a perfect set or a singleton.
\end{thrm}

\begin{proof} Let $A \subseteq P$ be universally measurable, and suppose that CPA$_{cube}(P)$ holds. Let $\mathcal{E}$ be the collection of perfect subsets of $P$ which are either contained in or disjoint from $A$. It suffices to see that $\mathcal{E}$ is $\mathcal{F}_{cube}$-dense. This follows almost immediately from the two lemmas above. Let $f \colon \Pi_{n \in \omega}\creals \to P$ be a continuous injection, and let $\mu$ be the Borel measure on $P$ defined by letting $\mu(I)$ be the measure of $f^{-1}[I]$ in the standard product measure on $\Pi_{n \in \omega}\creals$. Then there exist Borel subsets $B$, $N$ of $P$ such that $A \bigtriangleup B \subset N$ and $\mu(N) = 0$. Then
one of $B \setminus N$ and $(P \setminus B) \setminus N$ has positive $\mu$-measure, so by Lemma \ref{secondquote} there is a function $g \subseteq f$ as desired.
\end{proof}

In the model obtained by forcing over a model of the {\sf GCH} with an $\omega_{2}$-length countable support iteration of Sacks forcing, CPA$_{cube}(P)$ holds for every Polish space $P$, along with the equation $\mathfrak{c} = 2^{\aleph_{1}} = \aleph_{2}$ (see pages 52, 143-144 and 159-160 of \cite{CiPa}).

\section{Universal measurability and the Baire property}

A fundamental question about universally measurable sets, which remains open, was asked by Mauldin, Preiss and Weizs\"{a}cker \cite{MaPrWe} in
1983 :

\begin{ques}
Is it consistent that every universally measurable set of reals
has the Baire property?
\end{ques}

One could also ask whether it consistent that every universally measurable set of reals has universally null symmetric difference
with a set with the Baire property. The following theorem shows that these two questions are equivalent.

\begin{thrm}
\label{UMB-counter}
If $S$ is a universally measurable set of reals without the Baire property, then
$S \times \mathbb{R}$ is a universally measurable set which does not have
universally null symmetric difference with any set with the Baire property.
\end{thrm}


\begin{proof}
To see that $S \times \mathbb{R}$ is universally measurable, let
$\mu$ be a finite atomless Borel measure on $\mathbb{R} \times \mathbb{R}$.
Let $\nu$ be the projection measure on $\mathbb{R}$ given by
$\nu(I) = \mu(I \times \mathbb{R})$. Then $\mu(S \times
\mathbb{R}) = \nu(S)$.



Suppose towards a contradiction that $S\times\mathbb{R}$ has universally
null symmetric difference with $P\subseteq
\mathbb{R}\times\mathbb{R}$, and that $P$ has the property of Baire.
Since $S$ does not have the Baire property, there is an open interval
in which $S$ is neither meager nor comeager in any subinterval.
By restricting to this interval if necessary, we may assume that $S$ itself
is neither meager nor comeager in any interval.
Furthermore, we may assume without loss of generality that $P$ is meager or
comeager in $\mathbb{R}\times\mathbb{R}$, since we can pass to an open
rectangle where this is the case, and replace $S$ with its restriction
to the $x$-axis of the cube, where it still does not have the Baire property.

We now derive a contradiction to the fact that the symmetric
difference of $S\times\mathbb{R}$ and $P$ is universally null. We
shall use the fact that if $x\in \mathbb{R}$ and $Y\subseteq
\mathbb{R}$ is comeager, then $\{x\}\times Y$ is not universally
null. To see this note that $Y$ contains a perfect set, and so
there are non-atomic measures concentrating on $\{x\}\times Y$.

Suppose first that $P$ is meager in $\mathbb{R}\times\mathbb{R}$.
Then the set of $x\in\mathbb{R}$ so that $P_x=\{y\mid \langle
x,y\rangle\in P\}$ is meager, must be comeager. Since $S$ is not
meager it cannot be contained in the complement of this set. So
there is $x\in S$ with $P_x$ meager. Then
$(S\times\mathbb{R})\setminus P \supseteq
\{x\}\times(\mathbb{R}\setminus P_x)$, and since
$\mathbb{R}\setminus P_x$ is comeager, this set is not universally
null.


Suppose next that $P$ is comeager in $\mathbb{R}\times\mathbb{R}$.
Then the set of $x\in\mathbb{R}$ so that $P_x=\{y\mid \langle
x,y\rangle\in P\}$ is comeager, must be comeager. Since $S$ is not
comeager there must be $x$ in this set which does not belong to
$S$, namely there must be $x\not\in S$ with $P_x$ comeager. Then
$P\setminus (S\times\mathbb{R})\supseteq \{x\}\times P_x$, and
since $P_x$ is comeager, this set is not universally null.
\end{proof}

In most natural models one may consider, there are universally
measurable sets without the property of Baire.
First
let us note that in many cases the existence of such sets is a
consequence of the existence of a medial limit. A \emph{medial
limit} is a universally measurable function from
$\mathcal{P}(\omega)$ to $[0,1]$ which is finitely additive for
disjoint sets, and maps singletons to $0$ and $\omega$ to $1$. Godefroy and
Talagrand showed \cite{GoTa} that if $f$ is a medial limit, then
$f^{-1}[\{1\}]$ is a universally measurable filter without the property
of Baire (their argument can be easily modified to show that $f^{-1}[\{1\}]$ does
not have universally null symmetric difference with a set with the property of Baire, either).
Many models have medial limits. For
example, the cardinal invariant equation
$cov(\mathcal{M}) = \mathfrak{c}$ implies the existence of medial
limits (see \cite{F5}; this is a consequence of Martin's Axiom).
Theorem \ref{univrei} can be used to show, using Borel reinterpretations, that there is a medial limit in a
$\mathbb{B}(X)$ extension if there is one in the ground model.
The first author has recently shown \cite{La?} that consistently all universally measurable filters on $\omega$ have the property of Baire,
and therefore consistently there are no medial limits. A model in which there are no
medial limits is obtained by iterating super-perfect tree forcing (Miller forcing)
$\omega_2$ times, starting from a model of the {\sf CH}. In this model there is a universally
null set of reals of size $\aleph_{1}$ without the property of Baire \cite{LaBr}. Note that any universally null set with the property of Baire
must be meager, since otherwise it would contain a perfect set.

The argument below proves the existence of a nonmeager universally null set from the cardinal invariant equation $cov(\mathcal{M}) = cof(\mathcal{M})$. The invariant $cov(\mathcal{M})$ is the smallest cardinality of a collection of meager sets of reals whose union is the entire real line, while $cof(\mathcal{M})$ is the smallest cardinality of a collection of meager sets such that every meager set is contained in a member of the collection. It follows immediately from the definitions that $\aleph_{1} \leq cov(\mathcal{M}) \leq cof(\mathcal{M}) \leq \mathfrak{c}$. In the model obtained by the super-perfect iteration described in the previous paragraph, $cov(\mathcal{M}) = \aleph_{1}$ and $cof(\mathcal{M}) = \aleph_{2} = \mathfrak{c}$
(see \cite{Bl}).






\begin{thrm}
\label{UMB-counter-null}
Suppose that $cof(\mathcal{M}) = cov(\mathcal{M})$. Then there is
a nonmeager universally null set.
\end{thrm}

\begin{proof}
Let $\kappa = cof(\mathcal{M})$ and suppose that $\{D_{\alpha} :
\alpha < \kappa\}$ is a collection of meager sets such that every
meager set is contained in some $D_{\alpha}$. Pick reals
$x_{\alpha}$ $(\alpha < \kappa$) and countable sets $Y_{\alpha}$
$(\alpha < \kappa)$ such that:

\begin{itemize}
\item
each $Y_{\alpha}$ is a countable dense subset of $\mathbb{R}
\setminus (D_{\alpha} \cup \{ x_{\beta} : \beta < \alpha\})$;
\item
each $x_{\alpha} \in \mathbb{R} \setminus \bigcup_{\beta \leq
\alpha} (D_{\beta} \cup
Y_{\beta})$.
\end{itemize}
Note that the unions $\bigcup_{\beta \leq \alpha} (D_{\beta} \cup
Y_{\beta})$ do not cover any nonempty open subset of $\mathbb{R}$, since $\alpha<
cov(\mathcal{M})$. We use here the assumption that
$cov(\mathcal{M})=cof(\mathcal{M})$.

Then $\{x_{\alpha} : \alpha < \kappa\}$ is not meager, since it is
not contained in any $D_{\beta}$.

Also, $\{x_{\alpha} : \alpha < \kappa\}$ is not comeager in any
nonempty open set $O$, since for any meager set $M$ there is a
$\beta < \kappa$ such that $M \subseteq D_{\beta}$, and
$(Y_{\beta} \setminus D_{\beta}) \cap O$ is nonempty and
disjoint from $\{x_{\alpha} : \alpha < \kappa\}$.

Finally, since for any measure $\mu$ there is a comeager set $C$
such that $\mu(C)=0$ (see, for example \cite{O}), there is some $\beta<\kappa$ such that
$\mu(\mathbb{R} \setminus D_{\beta}) = 0$. So $\mu(\{ x_{\alpha} :
\alpha < \kappa\}) = 0$, since $\{ x_{\alpha} : \alpha < \beta\}$
is universally null by virtue of having cardinality less than
$\kappa=cov(\mathcal{M})\leq non({\cal N})$ (see \cite{BaJu} for
the last inequality; $non({\cal N})$ is the least cardinality of a
non-null set), and $\{ x_{\alpha} : \beta \leq \alpha < \kappa\}
\subseteq \mathbb{R} \setminus D_{\beta}$.
\end{proof}

Theorem \ref{UMB-counter-null} has been recently improved \cite{LaBr}, with the hypothesis weakened to the
assumption that $cof(\mathcal{M}) = \min \{ non(\mathcal{N}), \mathfrak{d}\}$, where
$\mathfrak{d}$ is the dominating number (see \cite{BaJu}).
We note that in the $\mathbb{B}(\mathfrak{c})$ extension of a model of the {\sf CH} (as well as the
corresponding iterated random real model, see \cite{Bl}), universally null sets of reals have size at most
$\aleph_{1}$, and sets of reals of size $\aleph_{1}$ are meager.

\noindent Department of Mathematics,
Miami University, Oxford, Ohio 45056, USA; Email: {\tt
larsonpb@muohio.edu} \vspace{\baselineskip}

\noindent Department of Mathematics,
University of California Los Angeles
Los Angeles, CA 90095-1555, USA; Email: {\tt
ineeman@math.ucla.edu}\vspace{\baselineskip}

\noindent The Hebrew University of Jerusalem, Einstein Institute of
Mathematics,\\ Edmond J. Safra Campus, Givat Ram, Jerusalem 91904,
Israel \vspace{\baselineskip}

\noindent Department of Mathematics, Hill Center - Busch
Campus, Rutgers, The State University of New Jersey, 110
Frelinghuysen Road, Piscataway, NJ 08854-8019, USA; Email: {\tt
shelah@math.huji.ac.il}

\end{document}